\documentclass{amsproc}
\usepackage[margin=1.2in,nomarginpar]{geometry}
\usepackage{amssymb}
\usepackage{amsmath}
\usepackage{mathdots}
\usepackage{amsbsy}
\usepackage{amscd}
\usepackage{amsthm}

\usepackage{url}
\usepackage{xcolor}
\usepackage{amsmath,amssymb,amsthm,amsfonts,longtable}
\usepackage[english]{babel}
\usepackage{tikz-cd}
\usetikzlibrary{cd}
\usetikzlibrary{decorations.markings}
\tikzset{negated/.style={
		decoration={markings,
			mark= at position 0.5 with {
				\node[transform shape] (tempnode) {$\times$};
			}
		},
		postaction={decorate}
	}
}
\usepackage{array}
\usepackage[colorlinks,citecolor=red,urlcolor=blue,bookmarks=false,hypertexnames=true]{hyperref} 
\newtheorem{theorem}{Theorem}

\newtheorem{lemma}[theorem]{Lemma}
\newtheorem{proposition}[theorem]{Proposition}

\newtheorem{remark}[theorem]{Remark}

\newcommand{\Irr}{\textnormal{Irr}}
\newcommand{\cd}{\textnormal{cd}}
\newcommand{\nl}{\textnormal{nl}}
\newcommand{\lin}{\textnormal{lin}}

\newcommand{\gal}{\textnormal{Gal}}
\newcommand{\aut}{\textnormal{Aut}}

\title[]{A Combinatorial Formula for the Wedderburn Decomposition of Rational Group Algebras of Split Metacyclic $p$-groups}
\author{Ram Karan Choudhary}
\address{Indian Institute of Technology, Bhubaneswar, Arugul Campus, Jatni, Khurda-752050, India.}
\email{ramkchoudhary1997@gmail.com}
\author{Sunil Kumar Prajapati$^*$}
\address{Indian Institute of Technology, Bhubaneswar, Arugul Campus, Jatni, Khurda-752050, India.}
\email{skprajapati@iitbbs.ac.in}
\thanks{$^{\textbf{*}}$ Corresponding author.
}
\begin{document}
	\maketitle

	\begin{abstract}
		In this article, we present a concise combinatorial formula for efficiently determining the Wedderburn decomposition of rational group algebra associated with a split metacyclic $p$-group $G$, where $p$ is an odd prime. We also provide a combinatorial formula to count irreducible rational representations of $G$ of distinct degrees. 
	\end{abstract}
	
	\section{Introduction}
	Let $p$ be an odd prime and let $G$ be a finite non-abelian split metacyclic $p$-group. Then $G$ can be presented in the form
\begin{equation}\label{prest:splitmetacyclic}
G = \langle a, b\mid a^{p^n} = b^{p^m} = 1, bab^{-1} = a^r \rangle,
\end{equation}
	where $n\geq 2, m \geq 1$, $(r, p^n) = 1$ and $r$ has multiplicative order $p^s$ modulo $p^n$ such that $1\leq s \leq \min \{n-1, m\}$.
	In this article, we study the Wedderburn decomposition of $\mathbb{Q}G$.
	 For various family of groups, the Wedderburn decompositions of rational group algebras have been extensively studied in \cite{ BG1, BG, BM14, Jes-Lea-Paq, ODRS04, Olt07, PW} and the authors used various concepts such as computation of the field of character values, Shoda pairs, numerical representation of cyclotomic algebras, etc, to compute the simple components of the rational group algebras. For a non-abelian split metacyclic $p$-group $G$ defined in \eqref{prest:splitmetacyclic}, we prove Theorem \ref{thm:wedderburnsplitmetacyclic}, which provides a combinatorial description for the Wedderburn decomposition of $\mathbb{Q}G$. Our result formulates the computation of the Wedderburn decomposition of a split metacyclic $p$-group $G$ solely based on the numerical values of $n, m, r ~ \text{and} ~ s$.
	\begin{theorem}\label{thm:wedderburnsplitmetacyclic}
		Let $p$ be an odd prime and let $\zeta_d$ be a primitive $d$-th root of unity. Consider a finite non-abelian split metacyclic $p$-group $G = \langle a, b\mid a^{p^n} = b^{p^m} = 1, bab^{-1} = a^r \rangle$, where $n\geq 2, m \geq 1$, $(r, p^n) = 1$ and $r$ has multiplicative order $p^s$ modulo $p^n$ such that $1\leq s \leq \min \{n-1, m\}$. Then we have the following.
		\begin{enumerate}
			\item {\bf Case ($n-s \geq m$).} In this case,
			$$\mathbb{Q}G \cong \mathbb{Q} \bigoplus_{\lambda=1}^m (p^\lambda+p^{\lambda-1})\mathbb{Q}(\zeta_{p^\lambda}) \bigoplus_{\lambda=m+1}^{n-s}p^m \mathbb{Q}(\zeta_{p^\lambda}) \bigoplus_{t=1}^s p^{m-t}M_{p^t}(\mathbb{Q}(\zeta_{p^{n-s}})).$$
			\item {\bf Case ($n-s < m$).} Suppose $m = (n-s)+k$. Then we have following two sub-cases.
			\begin{enumerate}
				\item {\bf Sub-case ($k \leq s$).} In this sub-case, 
				\begin{align*}
					\mathbb{Q}G \cong & \mathbb{Q} \bigoplus_{\lambda=1}^{n-s} (p^\lambda+p^{\lambda-1})\mathbb{Q}(\zeta_{p^\lambda}) \bigoplus_{\lambda=n-s+1}^{m}p^{n-s} \mathbb{Q}(\zeta_{p^\lambda}) \bigoplus_{t=1}^{k-1} p^{n-s}M_{p^t}(\mathbb{Q}(\zeta_{p^{n-s}}))\\ &\bigoplus_{t=1}^{k-1}\bigoplus_{\lambda=n-s+1}^{m-t} (p^{n-s}-p^{n-s-1})M_{p^t}(\mathbb{Q}(\zeta_{p^{\lambda}})) \bigoplus_{t=k}^s p^{m-t}M_{p^t}(\mathbb{Q}(\zeta_{p^{n-s}})).
				\end{align*}
				\item {\bf Sub-case ($k > s$).} In this sub-case, 
				\begin{align*}
					\mathbb{Q}G \cong & \mathbb{Q} \bigoplus_{\lambda=1}^{n-s} (p^\lambda+p^{\lambda-1})\mathbb{Q}(\zeta_{p^\lambda}) \bigoplus_{\lambda=n-s+1}^{m}p^{n-s} \mathbb{Q}(\zeta_{p^\lambda}) \bigoplus_{t=1}^s p^{n-s}M_{p^t}(\mathbb{Q}(\zeta_{p^{n-s}}))\\ &\bigoplus_{t=1}^s\bigoplus_{\lambda=n-s+1}^{m-t} (p^{n-s}-p^{n-s-1})M_{p^t}(\mathbb{Q}(\zeta_{p^{\lambda}})).
				\end{align*}
			\end{enumerate}
		\end{enumerate}
	\end{theorem}
	Several researchers have investigated complex irreducible representations and characters of metacyclic groups (see \cite{BGBasmaji, HK, HS, Munkholm}). In this article, we explicitly compute all inequivalent complex irreducible representations and characters of $G$ (defined in \eqref{prest:splitmetacyclic}) using the Wigner-Mackey method of little groups. In Section \ref{sec:rationalreprepabelian}, we give a brief review of irreducible rational representations of $C_{p^n} \times C_{p^m}$, which is required for the upcoming sections. Section \ref{sec:WMmethod} provides a concise introduction to the Wigner-Mackey method of little groups, while Section \ref{sec:complexrep} delves into the irreducible complex representations of $G$. Section \ref{sec:rationalrep} covers the results related to irreducible rational representations of $G$ which are crucial to prove Theorem \ref{thm:wedderburnsplitmetacyclic} in Section \ref{sec:mainresult}. In Theorem \ref{thm:rationalrep}, we give a combinatorial formula to count irreducible rational representations of $G$ of distinct degrees. 
	
	{\bf Notations.} We conclude this section by setting some notations, which are mostly standard. Throughout this paper, $p$ denotes an odd prime. For a finite group $G$, $Z(G)$, $G'$, and $\mathbb{Q}G$ represent its center subgroup, commutator subgroup, and the rational group algebra of $G$. Additionally, we employ $\Irr(G)$, $\lin(G)$, $\nl(G)$, and $\cd(G)$ to denote the sets of irreducible complex characters, linear complex characters, non-linear irreducible complex characters, and degrees of irreducible complex representations of G, respectively. For $\chi \in \Irr(G)$, we use $m_{\mathbb{Q}}(\chi)$ for the Schur index of $\chi$ over $\mathbb{Q}$, $\Omega(\chi) = m_{\mathbb{Q}}(\chi)\sum_{\sigma \in \gal(\mathbb{Q}(\chi) / \mathbb{Q})}\chi^{\sigma}$, and $\mathbb{Q}(\chi)$ for the field obtained by adjoining the values $\{\chi(g) : g\in G\}$ to $\mathbb{Q}$. The notations $M_{q}(D)$, $Z(B)$ and $\phi(q)$ signify a full matrix ring of order $q$ over the skewfield $D$, the center of an algebraic structure $B$ and  the Euler phi function, respectively.

	\section{Review of rational representations of $C_{p^n} \times C_{p^m}$}\label{sec:rationalreprepabelian}
	In this section, we derive a combinatorial description for counting irreducible rational representations of $C_{p^n} \times C_{p^m}$. Let $G = \langle a, b \mid  a^{p^n} = b^{p^m} = 1, ab = ba\rangle \cong C_{p^n} \times C_{p^m}$, where $n \geq m$. Then 
	$$\Irr(G) = \{\chi_{i,j} \mid \chi_{i,j}(a) = \zeta^i, \chi_{i,j}(b) = \omega^j, 0 \leq i \leq p^n-1, 0 \leq j \leq p^m-1\},$$
	where $\zeta$ is a primitive $p^n$-th root of unity and $\omega$ is a primitive $p^m$-th root of unity. In this section, $G$ always denotes the abelian group defined above. Now, we define an equivalence relation on $\Irr(G)$ by Galois conjugacy over $\mathbb{Q}$. We say that $\chi_{i,j}$ and $\chi_{i',j'}$ are Galois conjugates over $\mathbb{Q}$ if $\mathbb{Q}(\chi_{i,j}) = \mathbb{Q}(\chi_{i',j'})$ and there exists $\sigma \in$ Gal($\mathbb{Q}(\chi_{i,j}) / \mathbb{Q})$ such that $\chi_{i,j}^\sigma = \chi_{i',j'}$.
	\begin{lemma}\cite[Lemma 9.17]{I}\label{SC}
		Let $E(\chi)$ denotes the Galois conjugacy class over $\mathbb{Q}$ of a complex irreducible character $\chi$. Then 
		\[|E(\chi)| = [\mathbb{Q}(\chi) : \mathbb{Q}]. \]
	\end{lemma}
\noindent Now under the above set-up, we have the following remarks.

\begin{remark}\label{remark:rationalcounting_abelian}
\begin{enumerate}
\item 
	Suppose $E(\chi_{i, j})$ denotes the Galois conjugacy class of $\chi_{i, j}$ over $\mathbb{Q}$. Then $E(\chi_{0,0}) = \{\chi_{0,0}\}$ is the only Galois conjugacy class of size $1$.
	
\item {\bf Claim 1:} For each $r$ ($1 \leq r \leq m$), there are $2\sum_{k=0}^{r-1}\phi(p^k) + \phi(p^r)$ distinct Galois conjugacy classes of size $\phi(p^r)$.\\
{\bf Proof of Claim 1.}
Note that for a fixed $r$ and $0 \leq k \leq r$, we have $\mathbb{Q}(\chi_{p^{n-k}, p^{m-r}})= \mathbb{Q}(\chi_{p^{n-r}, p^{m-k}})=\mathbb{Q}(\zeta^{p^{n-r}})=\mathbb{Q}(\omega^{p^{m-r}})= \mathbb{Q}(\theta),$ where $\theta$ is a primitive $p^r$-th root of unity. Therefore,
	$$[\mathbb{Q}(\chi_{p^{n-k}, p^{m-r}}) : \mathbb{Q}] = [\mathbb{Q}(\chi_{p^{n-r}, p^{m-k}}) : \mathbb{Q}] = [\mathbb{Q}(\theta) : \mathbb{Q}] = \phi(p^r).$$
{\bf Case ($0 \leq k < r$).} In this case for each $k$ with $1\leq l <p^k$ and $(l,p)=1$, we have
	$$\gal(\mathbb{Q}(\chi_{lp^{n-k}, p^{m-r}}) / \mathbb{Q}) = \gal(\mathbb{Q}(\theta) / \mathbb{Q}) = \{\sigma : \sigma(\theta) = \theta^\alpha, 1 \leq \alpha < p^r, (\alpha, p^r) = 1\}.$$
Hence  $\chi_{lp^{n-k}, p^{m-r}}^\sigma(a)= (\zeta^{lp^{n-k}})^\sigma = \zeta^{\alpha lp^{n-k}}$ and $\chi_{lp^{n-k}, p^{m-r}}^\sigma(b)= (\omega^{p^{m-r}})^\sigma = \omega^{\alpha p^{m-r}}$. This implies that $\chi_{lp^{n-k}, p^{m-r}}^\sigma = \chi_{\alpha lp^{n-k}, \alpha p^{m-r}}$ for some $1\leq \alpha < p^r, (\alpha, p^r) = 1$. Thus, the Galois conjugacy class for the representative $\chi_{lp^{n-k}, p^{m-r}}$ is given by
	$$E(\chi_{lp^{n-k}, p^{m-r}}) = \{\chi_{\alpha lp^{n-k}, \alpha p^{m-r}} : 1\leq \alpha < p^r, (\alpha, p^r) = 1\}.$$
Further, let $l \neq l'$ where, $1 \leq l, l' < p^k$, $(l, p) = 1$ and $(l', p)=1$. Suppose there exits $\sigma \in \gal(\mathbb{Q}(\chi_{lp^{n-k}, p^{m-r}}) / \mathbb{Q})$ such that $\chi_{lp^{n-k}, p^{m-r}}^\sigma = \chi_{\alpha lp^{n-k}, \alpha p^{m-r}} = \chi_{l'p^{n-k}, p^{m-r}}$. Then $\sigma$ is identity map and so $\chi_{lp^{n-k}, p^{m-r}} = \chi_{l'p^{n-k}, p^{m-r}}$. Hence, $l = l'$, which is a contradiction. Therefore, there are $\phi(p^k)$ many distinct Galois conjugacy classes with representatives $\chi_{lp^{n-k}, p^{m-r}}$ (for $1\leq l< p^k$ and $(l, p) = 1$). By a similar discussion, again we get $\phi(p^k)$ many distinct Galois conjugacy classes with representatives $\chi_{p^{n-r}, lp^{m-k}}$ (for $1\leq l< p^k$ and $(l, p) = 1$).
Therefore, for $0 \leq k < r$, there are $2\phi(p^k)$ many Galois conjugacy classes which are explicitly given by
	$E(\chi_{lp^{n-k}, p^{m-r}})$ and $ E(\chi_{p^{n-r}, lp^{m-k}}),$
	where $1\leq l< p^k$ and $(l, p) = 1$.	\\
\noindent {\bf Case ($k=r$)}. In this case, there are $\phi(p^r)$ Galois conjugacy classes of size $\phi(p^r)$ and which are given by $E(\chi_{lp^{n-r}, p^{m-r}}),$
	where $1\leq l< p^r$ and $(l, p) = 1$.\\
	Therefore by the above discussion, for a fixed $r$ ($1 \leq r \leq m$), we get $2\sum_{k=0}^{r-1}\phi(p^k) + \phi(p^r)$ distinct Galois conjugacy classes of size $\phi(p^r)$. This completes the proof of Claim 1.
\item	{\bf Claim 2:} For each $r$ ($m < r \leq n$), there are $\sum_{k=0}^{m}\phi(p^k)$ distinct Galois conjugacy classes of size $\phi(p^r)$. \\
{\bf Proof of Claim 2.} By using the similar argument used in Remark \ref{remark:rationalcounting_abelian} (2), for a fixed $k$ ($0 \leq k \leq m$), we get $\phi(p^k)$ distinct Galois conjugacy classes, which are explicitly given by
	$E(\chi_{p^{n-r}, lp^{m-k}})$,
	where $1\leq l< p^k$ and $(l, p) = 1$. This completes the proof of Claim 2.
	\end{enumerate}
\end{remark}	
Now, Lemma \ref{lemma:counting} shows that the Galois conjugacy classes discussed in Remark \ref{remark:rationalcounting_abelian} exhaust all the irreducible complex characters of $G$.
	\begin{lemma}\label{lemma:counting}
		Let $n \geq m$ then
		\begin{equation}\label{eqn:sumformula}
		p^{n+m} = 1+ \sum_{r=1}^{m} \big[\phi(p^r)\cdot \big(2\sum_{k=0}^{r-1}\phi(p^k) + \phi(p^r)\big)\big] + \sum_{r=m+1}^{n}\big[\phi(p^r)\cdot \big(\sum_{k=0}^{m}\phi(p^k)\big)\big].
		\end{equation}		
	\end{lemma}
	\begin{proof} The first two term of the R.H.S. of \eqref{eqn:sumformula}
`		\begin{align*}
& 1+ \sum_{r=1}^{m}[\phi(p^r)\cdot(2\sum_{k=0}^{r-1}\phi(p^k) + \phi(p^r))] \\
&= 1
			+ \phi(p) [2\phi(1) + \phi(p)]
			+ \phi(p^2) [2\phi(1) + 2\phi(p) + \phi(p^2)] + \dots + \phi(p^m) [2\phi(1) + 2\phi(p)+ \dots +\phi(p^m)]\\
			& = 1 + [\phi^2(p) + \dots + \phi^2(p^m)]  + 2\phi(1)[\phi(p) + \dots + \phi(p^m)]
			+ 2\phi(p) [\phi(p^2) + \dots + \phi(p^m)]+ \\
			& \dots + 2\phi(p^{m-1})[\phi(p^m)]\\
			& = 1 + \frac{(p-1)(p^{2m}-1)}{p+1} + 2(p^m-1) + 2(p-1)(p^m-p) + 2p(p-1)(p^m - p^2)+ \\
			& \dots + p^{m-2}(p-1)(p^m - p^{m-1})\\
			& = 1 + \frac{(p-1)(p^{2m}-1)}{p+1} + 2(p^m-1) + 2(p-1)[p^m(1 + p + \dots + p^{m-2}) - (p + p^3 + \dots + p^{2m-3})]\\
			& = 1 + \frac{(p-1)(p^{2m}-1)}{p+1} + 2(p^m-1) + 2(p-1)[p^m \frac{p^{m-1}-1}{p-1} - \frac{p(p^{2m-2}-1)}{p^2-1}]\\
			& = 1 + \frac{(p-1)(p^{2m}-1)}{p+1} + 2(p^m-1) + 2(p^{2m-1}-p^m) - 2\frac{p(p^{2m-2}-1)}{p+1}\\
			& = 2p^{2m-1} - 1 + \frac{p^{2m+1}- p^{2m} - 2p^{2m-1}+ p + 1}{p+1}\\
			& = 2p^{2m-1} - 1 - 2p^{2m-1} + p^{2m} + 1\\
			& = p^{2m}.
		\end{align*}
		Now the second term of the R.H.S. of \eqref{eqn:sumformula}
	\begin{align*}
\sum_{r=m+1}^{n}[\phi(p^r)\cdot(\sum_{k=0}^{m}\phi(p^k))] 
= \sum_{r=m+1}^{n}[\phi(p^r)\cdot p^m]
= (p^{n}-p^m) p^m
= p^{n+m}-p^{2m}.
		\end{align*}		
This completes the proof. 		
		\end{proof}
Now, since $G$ is an abelian group, the Schur index $m_{\mathbb{Q}}(\chi_{i, j}) = 1$ (for all $0 \leq i \leq p^n-1, 0 \leq j \leq p^m-1$). Schur's theory implies that there exists a unique irreducible $\mathbb{Q}$-representation $\rho$ of $G$ such that $\chi_{i,j}$ occurs as an irreducible constituent of $\rho \otimes_{\mathbb{Q}}\mathbb{F}$ with multiplicity 1, where $\mathbb{F}$ is a splitting field of $G$. Hence, the distinct Galois conjugacy class gives the distinct rational representation of $G$. We summarize the above discussion in Proposition \ref{prop:Ab}. 
	\begin{proposition}\label{prop:Ab}
		Let $G = \langle a, b \mid  a^{p^n} = b^{p^m} = 1, ab = ba\rangle \cong C_{p^n} \times C_{p^m}$, where $n \geq m$. Then we have the following.
		\begin{enumerate}
			\item $G$ has only one irreducible rational representation of degree 1.
			\item For each $r$ ($1 \leq r \leq m$), $G$  has $2\sum_{k=0}^{r-1}\phi(p^k) + \phi(p^r) = p^r + p^{r-1}$ many inequivalent irreducible rational representations of degree $\phi(p^r)$.
			\item  For each $r$ ($m < r \leq n$), $G$ has  $\sum_{k=0}^{m}\phi(p^k) = p^m$ many inequivalent irreducible rational representations of degree $\phi(p^r)$.
			\end{enumerate}
	\end{proposition}
	\noindent		Hence, from Proposition \ref{prop:Ab}, $G$ has total $1+\sum_{r=1}^{m}(p^r+p^{r-1})+(n-m)p^m=\sum_{k=0}^{m}(n+m+1-2k)\phi(p^k)$ inequivalent irreducible rational representations.

	\section{Wigner-Mackey method}\label{sec:WMmethod}
	 Let $G$ be semi-direct product of a normal abelian subgroup $N$ by a subgroup $H$. Since $G$ acts on $N$ by conjugation, this action induces an action on $\Irr(N)$ as follows:
	\begin{center}
		$\chi\in\Irr(N)$, $g\in G$ and $\forall a\in N$, we have $(g\cdot \chi)(a) = \chi(gag^{-1})$.
	\end{center}
	By the definition of conjugate representations, $g\cdot \chi$ is the same as the conjugate representation $\chi^g$ of $\chi$ by $g\in G$. Let $I_\chi$ be the inertia group of $\chi$ in $G$. Observe that $N\leq I_\chi$. Let $H_\chi := I_\chi\cap H$. Then $I_\chi$ is a semi-direct product of $N$ by $H_\chi$. It can be shown that each $\chi\in \Irr(N)$ can be extended to a homomporphism from $I_\chi$ to $\mathbb{C}^*$, in such a way that this extended homomorphism is the trivial map when restricted to $H_\chi$. So we can regard $\chi$ as a linear representation of $I_\chi$. Further, if $\rho$ is a complex irreducible representation of $H_\chi$ and since $I_{\chi}/N \cong H_{\chi}$, then $\rho$ is an irreducible complex representation of $I_\chi$. Therefore, the representation $\chi\otimes\rho$ is an irreducible complex representation of $I_{\chi}$ and $\deg (\chi\otimes\rho) =\deg \rho$.
	\begin{theorem}\cite[Theorem 4.7.7]{CM}\label{theorem:WMmethod}
		Let $G$ be a finite group which is a semi-direct product of an abelian group $N$ by a subgroup $H$. Let $\mathcal{O}_1, \mathcal{O}_2, \dots, \mathcal{O}_t$ be the distinct orbits under the action of $G$ on $\Irr(N)$ and let $\chi_j$ be a representative of the orbit $\mathcal{O}_j$. Let $I_j$ denote the inertia group of $\chi_j$ and let $H_j := I_j\cap H$. For any irreducible representation $\rho$ of $H_j$, let $\theta_{j, \rho}$ denote the representation of $G$ induced from the irreducible representation $\chi_j\otimes\rho$ of $I_j$. Then, we have
		\begin{enumerate}
			\item $\theta_{j, \rho}$ is irreducible,
			\item If $\theta_{j, \rho}$ and $\theta_{j', \rho'}$ are isomorphic, then $j = j'$ and $\rho$ is isomorphic to $\rho'$,
			\item Every irreducible representation of $G$ is isomorphic to one of the $\theta_{j, \rho}$.		
		\end{enumerate}
	\end{theorem}
	
	\section{Complex representations of split metacyclic $p$-groups}\label{sec:complexrep}
	Basmaji \cite{BGBasmaji} obtained complex irreducible representations of metacyclic groups. However, in this section we will first describe complex irreducible representations a finite non-abelian split metacyclic $p$-group $G$ (defined in \eqref{prest:splitmetacyclic}), in detail by using Wigner-Meckey method of little groups. Observe that $G =N \rtimes_{T} H$ for some homomorphism $T : H \rightarrow \aut(N)$, where $N = \langle a \rangle \cong C_{p^n}$ and $H = \langle b \rangle \cong C_{p^m}$. Here,
	$$ \Irr(N)= \{\chi_i : \chi_i(a) = \zeta^i, i = 0, 1, 2, \dots, p^n-1\}, $$
	where $\zeta$ is a primitive $p^n$-th root of unity. Now, $G$ acts on $\Irr(N)$ by conjugation with the subgroup $N$ acting trivially. So, this induces an action of $G/N$ on $\Irr(N)$. As $\chi_k^b(a) = \chi_k(bab^{-1}) = \chi_k(a^r) = \chi_{rk}(a)$, so $\chi_k^b = \chi_{rk}$. Similarly, $\chi_k^{b^2} = \chi_{r^2k}$, $\chi_k^{b^3} = \chi_{r^3k}$ and so on. Also, $\chi_k^{b^{p^s}} = \chi_{r^{p^s}k} = \chi_k$. Thus, no orbit has size greater than $p^s$.
	\begin{lemma}\label{H}
		Let $p$ be an odd prime. Suppose $a$ is an integer co-prime to $p$ and $a$ has multiplicative order $f$ $\mod p$. For each non-zero integer $x$, let $w_p(x) = \max \{l : p^l \mid x\}$. If $f \mid m$, then $w_p(a^m-1) = w_p(a^f-1) + w_p(m)$.
		\begin{proof}
			See \textnormal{\cite[Lemma 8.1]{BH}}.
		\end{proof}
	\end{lemma}
	\begin{lemma}\label{NT}
		Let $p$ be a prime and $k\in \mathbb{N}$ such that $(k, p) = 1$. Then for a fixed natural numbers $s$ and $n$ with $s< n$, the multiplicative order of $(kp^{n-s} + 1)
		$ is $p^s \mod p^n$.
	\end{lemma}
	\begin{proof}

		Choose $a=kp^{n-s}+1$ and $m=p^s$ in Lemma \ref{H}.
	\end{proof}
	Under the above set-up, we have the following remarks.
\begin{remark}\label{remark:splitmeta}
\begin{enumerate}
\item For some prime $p$ and $n\in \mathbb{N}$, consider $x\in \mathbb{Z}_{p^n}$. Then observe that there exist unique positive integers $k, s$ such that $x=1+kp^{n-s}$, where $k< p^s$, $s\leq n$ and $(k,p^s)=1$. 
	 Hence from lemma \ref{NT}, in the presentation of $G$ (see \eqref{prest:splitmetacyclic}), we have $r = kp^{n-s} + 1$ for some $k$ such that $1\leq k< p^s$ and $(k, p) = 1$. Therefore, by the above discussion, it is easy to check that
		$\{\chi_0\}, \{\chi_{p^s}\}, \{\chi_{2p^s\}, \dots, \{\chi_{(p^{n-s}-1)}p^s}\}$
	are the only orbits of size one under the above defined action of $G$ on $\Irr(N)$. Hence, the total number of orbits of size one is $p^{n-s}$. 	Moreover, the inertia group of $\chi_{\lambda p^s}, 0 \leq \lambda \leq p^{n-s}-1$ is $G$ itself.	Corresponding to the orbit $\mathcal{O}_{\chi_{\lambda p^s}} = \{\chi_{\lambda p^s}\}$ (for $0 \leq \lambda \leq p^{n-s} -1$) of size 1, the irreducible complex representations of $G$ are given by
		\begin{equation}\label{eq:linearcomplexrep}
			\theta_{\chi_{\lambda p^s}, \omega}(a) = \zeta^{\lambda p^s} \quad \text{and} \quad \theta_{\chi_{\lambda p^s}, \omega}(b) = \omega,
		\end{equation}
	where $\zeta$ is any fixed primitive $p^n$-th root of unity and $\omega$ is a $p^m$-th root of unity (by Theorem \ref{theorem:WMmethod}). In this manner, we obtain $p^{n+m-s}=|G/G{}'|$ many degree 1 representations of $G$.
	
\item	\noindent {\bf Claim:}\label{claim1} For a fixed $t \leq s$, we have
		$r^{p^t}p^{s-t} \equiv p^{s-t}(\mod p^n)$, where $r$ is given by \eqref{prest:splitmetacyclic}.\\
	\noindent {\bf Proof of the claim:}	In view of the above discussion, $r=1+kp^{n-s}$ and therefore, by Lemma \ref{NT}, we get 
$(kp^{n-s} + 1)^{p^t} \equiv 1 (\mod p^{n-s+t})$. Hence $r^{p^t}p^{s-t} \equiv p^{s-t}(\mod p^n)$. This completes the proof of the claim.

\noindent Therefore, by the above claim, for  a fixed $t$ ($1 \leq t\leq  s$), the orbit of $\chi_{lp^{s-t}}$ is $$\{\chi_{lp^{s-t}}, \chi_{rlp^{s-t}}, \dots, \chi_{r^{p^t-1}lp^{s-t}} \},$$ 
where $1\leq l < p^{n-s+t}$ with $(l,p)=1$. Hence, for a fixed $t$ ($1 \leq t \leq s$), we can see that, there are 
		$\phi(p^{n-s+t})/p^t = p^{n-s} - p^{n-s-1} = \phi(p^{n-s})$
	orbits of size $p^t$. Hence, the total number of orbits of size greater than $1$ is $s\phi(p^{n-s})$. Moreover, the inertia group of $\chi_{lp^{s-t}}$ is $\langle a, b^{p^t} \rangle$.

\item For a fixed $t$ ($1 \leq t \leq s$), we have $H_{\chi_{lp^{s-t}}}=\langle b^{p^t} \rangle \cong C_{p^{m-t}}$, where $1\leq l < p^{n-s+t}$ with $(l,p)=1$. Therefore, by Theorem \ref{theorem:WMmethod}, there are $p^{m-t}$ irreducible complex representations of $G$ corresponding to each orbit $\mathcal{O}_{\chi_{lp^{s-t}}} = \{\chi_{lp^{s-t}}, \chi_{rlp^{s-t}}, \dots, \chi_{r^{p^t-1}lp^{s-t}} \}$. Since $|G/I_{\chi_{lp^{s-t}}}|=p^t$, they are of degree $p^t$. Moreover, these $p^{m-t} $ irreducible complex representations of $G$ are given explicitly as follows:
	\begin{equation}\label{eq:non-linearcomplexrep}
		\begin{aligned}
			\theta_{\chi_{lp^{s-t}}, \omega}(a) &= \left(
			\begin{array}{ccccc}
				\zeta^{lp^{s-t}} & 0 & 0 & \cdots & 0\\
				0 & \zeta^{rlp^{s-t}} & 0 & \cdots & 0\\
				0 & 0 & \zeta^{r^2lp^{s-t}} & \cdots & 0\\
				\vdots & \vdots & \vdots & \ddots \\
				0 & 0 & 0 & \cdots & \zeta^{r^{p^t-1}lp^{s-t}}
			\end{array}\right) \quad \text{and}\\
			  \theta_{\chi_{lp^{s-t}}, \omega}(b) &= \left(
			\begin{array}{ccccc}
				0 & 1 & 0 & \cdots & 0\\
				0 & 0 & 1 & \cdots & 0\\
				\vdots & \vdots & \vdots & \ddots\\
				0 & 0 & 0 & \cdots & 1\\
				\omega & 0 & 0 & \cdots & 0
			\end{array}\right),
		\end{aligned}
	\end{equation}
	where $\zeta$ is any fixed primitive $p^n$-th root of unity, $\omega$ is a $p^{m-t}$-th root of unity.\\
	\end{enumerate}
	\end{remark}
	We summarize the above discussion in Proposition \ref{prop:complexrep}.
	\begin{proposition}\label{prop:complexrep}
		Let $p$ be an odd prime. Consider the group $G$ defined in \eqref{prest:splitmetacyclic}. Then, we have the following.
		\begin{enumerate}
			\item $G$ has $p^{n+m-s}$ many complex representations of degree $1$, and they are explicitly given in \eqref{eq:linearcomplexrep}.
			\item For a given $t$ ($1 \leq t \leq s$), $G$ has $\phi(p^{n-s})p^{m-t}$ many inequivalent irreducible complex representations of degree $p^t$ and they are explicitly given in \eqref{eq:non-linearcomplexrep}.
			\item $G$ has total $p^{n+m-s} + p^{n+m-s-1} - p^{n+m-2s-1}$ many inequivalent irreducible complex representations.
		\end{enumerate}
	\end{proposition}
	\begin{proof}
		\begin{enumerate}
			\item See Remark \ref{remark:splitmeta} (1).
			\item For a fixed $t$ ($1 \leq t \leq s$), the number of orbits of size $p^t$ is $\phi(p^{n-s})$ and hence by Remark \ref{remark:splitmeta} (2), we get the result.
			\item Since \[|G| = p^{n+m-s} + \phi(p^{n-s})\sum_{t=1}^{s}p^{2t}p^{m-t}\]
			and hence the total number of inequivalent irreducible complex representations is $p^{n+m-s} + \phi(p^{n-s})\sum_{t=1}^{s}p^{m-t}$.
		\end{enumerate}
	\end{proof}

	\section{Rational representations of split metacyclic $p$-groups}\label{sec:rationalrep}
	In this section, we discuss the irreducible rational representations of 
a split metacyclic $p$-groups $G$ (defined in \eqref{prest:splitmetacyclic}). 	
To find all the inequivalent irreducible rational representations of $G$, we divide the set of all inequivalent irreducible $\mathbb{Q}$-representations of $G$ into two parts:
	\begin{enumerate}
		\item the irreducible $\mathbb{Q}$-representations whose kernels contain $G'$ and
		\item the irreducible $\mathbb{Q}$-representations whose kernels do not contain $G'$.
	\end{enumerate}	
\noindent Since $G/G' = \langle a^{p^s}, b \rangle \cong C_{p^{n-s}}\times C_{p^m}$, from Proposition \ref{prop:Ab}, we have one of the following two cases for rational irreducible representations of $G$ whose kernels contain $G'$, which is covered in the Remark \ref{remark:linear_metacyclic}.
\begin{remark}\label{remark:linear_metacyclic}	{\bf Case 1 ($n-s \geq m$).} In this case, $G$ has 1 for $\lambda = 0$, $p^{\lambda -1}(p+1)$ for $1 \leq \lambda \leq m$ and $p^m$ for each $\lambda$ with $m < \lambda \leq n-s$ many inequivalent irreducible rational representations of degree $\phi(p^\lambda)$.\\
	{\bf Case 2 ($n-s < m$).} In this case, $G$ has 1 for $\lambda = 0$, $p^{\lambda -1}(p+1)$ for $1 \leq \lambda \leq n-s$ and $p^{n-s}$  for each $\lambda$ with $n-s < \lambda \leq m$ many inequivalent irreducible rational representations of degree $\phi(p^\lambda)$.\\
	Hence, $G$ has $\sum_{k=0}^{\min\{n-s, m\}}(n+m+1-s-2k)\phi(p^k)$ many inequivalent irreducible rational representations whose kernels contain $G'$ (see Proposition \ref{prop:Ab}).
\end{remark}	
   Now, we will compute all those irreducible rational representations of $G$ whose kernels do not contain $G'$.  We have $\cd(G)= \{p^t : 0 \leq t \leq s\}$. For a fixed $t$ ($1 \leq t \leq s$), there are $\phi(p^{n-s})p^{m-t}$ many irreducible complex representations of degree $p^t$ of $G$ (see Proposition \ref{prop:complexrep}), namely, $\theta_{\chi_{lp^{s-t}}, \omega}$ (see \eqref{eq:non-linearcomplexrep}). Let $\psi$ denote the characters corresponding to the representations $\theta_{\chi_{lp^{s-t}}, \omega}$. We have 
   \begin{equation}\label{eq:charofb}
   	\theta_{\chi_{lp^{s-t}}, \omega}(b) = \left(
   	\begin{array}{ccccc}
   		0 & 1 & 0 & \cdots & 0\\
   		0 & 0 & 1 & \cdots & 0\\
   		\vdots & \vdots & \vdots & \ddots\\
   		0 & 0 & 0 & \cdots & 1\\
   		\omega & 0 & 0 & \cdots & 0
   	\end{array}\right),
   \end{equation}
   where $\omega$ is a $p^{m-t}$-th root of unity. Observe that for any non-negative integer power of $\theta_{\chi_{lp^{s-t}}, \omega}(b)$ is either a diagonal matrix or its diagonal entries are zero. In particular, for a non-negative integer $\alpha$, $\theta_{\chi_{lp^{s-t}}, \omega}(b^{\alpha p^t}) = (\theta_{\chi_{lp^{s-t}}, \omega}(b))^{\alpha p^t} = \omega^\alpha I_{p^t}$, where $I_{p^t}$ is the identity matrix of order $p^t$. Thus, $\psi(b^{\alpha p^t}) = p^t\omega^\alpha$ and $\psi(b^\beta) = 0$ whenever $p^t \nmid \beta$.
   Now to find out the value of $\psi(a^i)$ (for some $i$), we first prove Lemma \ref{lemm:sumprimitveroots}.
   \begin{lemma}\label{lemm:sumprimitveroots}
   	Let $M, S\in \mathbb{N}$ and let $\zeta$ be a primitive $p^M$-th root of unity, where $p$ odd prime, $0 \leq S < M$. Then
   	\[\sum_{i=0}^{p^S-1}\zeta^{(1+kp^{M-S})^i} = 0,\]
   	where $(k, p)=1$ and $0\leq k < p^S$.
   \end{lemma}
   \begin{proof}
As $k$ is co-prime to $p$, $\zeta^k$ is also a primitive $p^M$-th root of unity. Since $M>S$, $\zeta^{kp^{M-S}}$ is a primitive $p^S$-th root of unity so that
   	\[1+ \zeta^{kp^{M-S}} + \zeta^{2kp^{M-S}} + \cdots + \zeta^{(p^S-1)kp^{M-S}} = 0 \]
   	Multiplying by $\zeta$, we have 
   	\[\zeta + \zeta^{1+kp^{M-S}} + \zeta^{1+2kp^{M-S}} + \cdots + \zeta^{1+(p^S-1)kp^{M-S}} = 0 \]
   	By Lemma \ref{NT}, the order of $1+kp^{M-S}$ mod $p^M$ is $p^S$. Thus, $|\{(1+kp^{M-S})^i\mod p^M~:~ 0\leq i\leq p^S-1\}|=p^S$. Further, 
$$\{(1+kp^{M-S})^i\mod p^M~:~ 0\leq i\leq p^S-1\}=\{(1+i{'}kp^{M-S})\mod p^M~:~ 0\leq i{'}\leq p^S-1\}.$$   	
   	
Therefore, we can rewrite the above expression as
   	\[\zeta + \zeta^{1+kp^{M-S}} + \zeta^{(1+kp^{M-S})^2} +\cdots + \zeta^{(1+kp^{M-S})^{p^S-1}} = 0, \]
   	and hence we get the required identity.
   \end{proof}
   Further, we have 
   \begin{equation}\label{eq:charofa}
   	\theta_{\chi_{lp^{s-t}}, \omega}(a) = \left(
   	\begin{array}{ccccc}
   		\zeta^{lp^{s-t}} & 0 & 0 & \cdots & 0\\
   		0 & \zeta^{rlp^{s-t}} & 0 & \cdots & 0\\
   		0 & 0 & \zeta^{r^2lp^{s-t}} & \cdots & 0\\
   		\vdots & \vdots & \vdots & \ddots \\
   		0 & 0 & 0 & \cdots & \zeta^{r^{p^t-1}lp^{s-t}}
   	\end{array}\right),
   \end{equation}
   	where $\zeta$ is a primitive $p^n$-th root of unity. Observe that $\zeta^{lp^{s-t}}$ is a primitive $p^{n-s+t}$-th root of unity. Now by using Lemma  \ref{lemm:sumprimitveroots} with $M=n-s+t$ and $S=t$, we get
	$$\psi(a)=\sum_{i=0}^{p^t-1}\zeta^{r^ilp^{s-t}} = \sum_{i=0}^{p^t-1}\zeta^{(1+kp^{n-s})^ilp^{s-t}} = 0.$$
Moreover, $\psi(a^{p^t}) = p^t\zeta^{lp^s}$. It is routine to check that
	$$\psi(a^i) = \begin{cases}
		p^t\zeta^{ilp^{s-t}}, &\quad \text{ if } p^t\mid i,\\
		0            &\quad \text{ otherwise. }\\
	\end{cases}$$
	Hence, we can compute the character value, $\psi(a^ib^j)$ of all the elements of $G$ as follows:
	\begin{equation}\label{eq:chatractervalue}
		\psi(a^ib^j) = \begin{cases}
			p^t\omega^{j_1}\zeta^{ilp^{s-t}},  &\quad \text{ if } p^t\mid i ~ \text{and} ~ j=j_1p^t,\\
			0            &\quad \text{ otherwise. }\\
		\end{cases}
	\end{equation}
Now, suppose $\omega$ is a $p^\lambda$-th primitive root of unity in \eqref{eq:non-linearcomplexrep}, where $0 \leq \lambda \leq m-t$. Then we have
     $$\gal(\mathbb{Q}(\psi) / \mathbb{Q}) = \gal(\mathbb{Q}(\eta) / \mathbb{Q}),$$
	where, $\eta$ is a primitive $p^{n-s}$-th root of unity for $n-s \geq \lambda$ and $\eta$ is a primitive $p^\lambda$-th root of unity for $n-s < \lambda$.\\
	 Hence, by the above discussion and from Lemma \ref{SC}, we have the following remark, which summarizes the counting of Galois conjugacy classes of those irreducible complex characters of $G$ whose kernels do not contains $G{}'$.
	
\begin{remark}\label{remark:galoisclasses_splitmeta} 
	 For a fixed $t$ ($1\leq t \leq s$), we have one of the following two cases.\\
	 {\bf Case 1 ($n-s \geq m-t$).} There are $\sum_{\lambda = 0}^{m-t}\phi(p^\lambda) = p^{m-t}$ many distinct Galois conjugacy classes of size $\phi(p^{n-s})$.\\
	 {\bf Case 2 ($n-s < m-t$).} There are $\sum_{\lambda = 0}^{n-s}\phi(p^\lambda) = p^{n-s}$ many distinct Galois conjugacy classes of size $\phi(p^{n-s})$ and $\phi(p^{n-s})$ many distinct Galois conjugacy classes of size $\phi(p^\lambda)$ for $n-s < \lambda \leq m-t$.\\
\end{remark}	 
	 Let $G$ be a $p$-group and $\chi \in \Irr(G)$. If $p$ is an odd prime, then  $m_\mathbb{Q}(\chi) = 1$, otherwise $m_\mathbb{Q}(\chi)\in \{1,2\}$ (see \cite[Corollary 10.14]{I}). Hence, Schur's theory implies that there exists a unique irreducible $\mathbb{Q}$-representation $\rho$ of $G$ such that 
	 $$\rho = \bigoplus_{\sigma \in \gal(\mathbb{Q}(\psi) / \mathbb{Q})} (\theta_{\chi_{lp^{s-t}}, \omega})^\sigma$$
and $\deg \rho =\psi(1)|\gal(\mathbb{Q}(\psi) / \mathbb{Q})|$. Hence, the number of irreducible rational representations of $G$ is equal to the number of Galois conjugacy classes of $\Irr(G)$. In view of Remark \ref{remark:linear_metacyclic} and Remark \ref{remark:galoisclasses_splitmeta}, we prove Theorem \ref{thm:rationalrep}, which provides the counting of irreducible rational representations of a split metacyclic $p$-group.
	\begin{theorem}\label{thm:rationalrep}
		Let $p$ be an odd prime. Consider a finite non-abelian split metacyclic $p$-group $G = \langle a, b\mid a^{p^n} = b^{p^m} = 1, bab^{-1} = a^r \rangle$, where $n\geq 2, m \geq 1$, $(r, p^n) = 1$ and $r$ has multiplicative order $p^s$ modulo $p^n$ such that $1\leq s \leq \min \{n-1, m\}$. Then we have the following.
		\begin{enumerate}
			\item {\bf Case ($n-s \geq m$).} In this case, $G$ has 1 for $\lambda = 0$, $p^{\lambda - 1}(p+1)$ for $1 \leq \lambda \leq m$, $p^m$ for each $\lambda$ with $m < \lambda \leq n-s$ and $p^{m-t}$ for $\lambda = n-s+t$ with $1 \leq t \leq s$ many inequivalent irreducible rational representations of degree $\phi(p^\lambda)$.
			\item {\bf Case ($n-s < m$).} Suppose $m = (n-s)+k$. In this case, we have two sub-cases.
			\begin{enumerate}
				\item {\bf Sub-case ($k \leq s$).} In this sub-case, $G$ has 1 for $\lambda = 0$, $p^{\lambda - 1}(p+1)$ for $1 \leq \lambda \leq n-s$, $2p^{n-s} + (t-1)\phi(p^{n-s})$ for $\lambda = n-s+t$ with $1 \leq t \leq k-1$, $p^{n-s}+(t-1)\phi(p^{n-s})+p^{m-t}$ for $\lambda=n-s+t$ with $t=k$ and $p^{m-t}$ for $\lambda = n-s+t$ with $k+1 \leq t \leq s$ many inequivalent irreducible rational representations of degree $\phi(p^\lambda)$.
				\item {\bf Sub-case ($k > s$).} In this sub-case, $G$ has 1 for $\lambda = 0$, $p^{\lambda - 1}(p+1)$ for $1 \leq \lambda \leq n-s$, $2p^{n-s} + (t-1)\phi(p^{n-s})$ for $\lambda = n-s+t$ with $1 \leq t \leq s$ and $p^{n-s}+s\phi(p^{n-s})$ for each $\lambda$ ($n+1\leq \lambda \leq m$) many inequivalent irreducible rational representations of degree $\phi(p^\lambda)$.
			\end{enumerate}
		\end{enumerate}
	\end{theorem}
\begin{proof}
\begin{enumerate}
\item In the case when $n-s\geq m$, we get 
$n-s> m-t$ for $1\leq t\leq s$. Then the proof of Theorem \ref{thm:rationalrep}(1) follows from Case 1 of Remark \ref{remark:linear_metacyclic} and Case 1 of Remark \ref{remark:galoisclasses_splitmeta}.
\item In this case, let $m=n-s+k$. 
\begin{enumerate}
\item Suppose $k\leq s$. In this sub-case $m\leq n$. Further, $n-s< m-t$ for $1\leq t\leq k-1$ and 
$n-s\geq m-t$ for $k\leq t\leq s$. Now, from Case 2 of Remark \ref{remark:linear_metacyclic}, $G$ has 1 for $\lambda = 0$ and $p^{\lambda - 1}(p+1)$ for $1 \leq \lambda \leq n-s$ many inequivalent irreducible rational representations of degree $\phi(p^\lambda)$. Again, from Case 2 of Remark \ref{remark:linear_metacyclic} and Case 2 of Remark \ref{remark:galoisclasses_splitmeta}, $G$ has  $2p^{n-s} + (t-1)\phi(p^{n-s})$ for $\lambda = n-s+t$ with $1 \leq t \leq k-1$  many inequivalent irreducible rational representations of degree $\phi(p^\lambda)$. Next, if $t=k$, then from Case 2 of Remark \ref{remark:linear_metacyclic}, Case 1 and Case 2 of Remark \ref{remark:galoisclasses_splitmeta}, $G$ has $p^{n-s}+(t-1)\phi(p^{n-s})+p^{m-t}$  many inequivalent irreducible rational representations of degree $\phi(p^{n-s+t}=p^m)$. Finally, from Case 1 of Remark \ref{remark:galoisclasses_splitmeta}, for $\lambda = n-s+t$ with $k+1 \leq t \leq s$, $G$ has $p^{m-t}$  many inequivalent irreducible rational representations of degree $\phi(p^\lambda)$.

\item Suppose $k>s$. In this sub-case, $m>n$. Further, $n-s< m-t$ for $1\leq t\leq s$. 
Here, from Case 2 of Remark \ref{remark:linear_metacyclic}, $G$ has 1 for $\lambda = 0$ and $p^{\lambda - 1}(p+1)$ for $1 \leq \lambda \leq n-s$ many inequivalent irreducible rational representations of degree $\phi(p^\lambda)$. Again, from Case 2 of Remark \ref{remark:linear_metacyclic} and Case 2 of Remark \ref{remark:galoisclasses_splitmeta}, $G$ has  $2p^{n-s} + (t-1)\phi(p^{n-s})$ for $\lambda = n-s+t$ with $1 \leq t \leq s$  many inequivalent irreducible rational representations of degree $\phi(p^\lambda)$. Finally, from Case 2 of Remark \ref{remark:linear_metacyclic} and Case 2 of Remark \ref{remark:galoisclasses_splitmeta}, for each $\lambda$ with  $n+1\leq \lambda \leq m$, $G$ has $p^{n-s}+s\phi(p^{n-s})$  many inequivalent irreducible rational representations of degree $\phi(p^\lambda)$.
\end{enumerate}
This completes the proof of Theorem \ref{thm:rationalrep}(2).
\end{enumerate}
\end{proof}

	\begin{remark}
		\begin{enumerate}
				\item Let $p$ be an odd prime and $G = \langle a, b\mid a^{p^n} = b^{p^m} = 1, bab^{-1} = a^r \rangle$, where $n\geq 2, m \geq 1$, $(r, p^n) = 1$ and $r$ has multiplicative order $p^s$ modulo $p^n$ such that $1\leq s \leq \min \{n-1, m\}$. From Lemma \ref{NT}, $r = kp^{n-s} + 1$ for some $(k, p) = 1$. Humphries and Skabelund \cite[Theorem 1.1]{HS} proved that any two split metacyclic groups with the same character tables are isomorphic. Since character values of $G$ corresponding to an irreducible complex representation are independent of $k$ (see \eqref{eq:chatractervalue}),
				$$\langle a, b\mid a^{p^n} = b^{p^m} = 1, bab^{-1} = a^{1+kp^{n-s}} \rangle \cong \langle a, b\mid a^{p^n} = b^{p^m} = 1, bab^{-1} = a^{1+p^{n-s}} \rangle.$$
				
				\item For an odd prime $p$, up to isomorphism, there are total $\min \{n-1, m\}$ distinct non-abelian groups which are isomorphic to $C_{p^n} \rtimes_{T} C_{p^m}$ for some homomorphism $T : C_{p^m} \rightarrow \aut(C_{p^n})$.
		\end{enumerate}
	\end{remark}
	
	\section{Rational group algebra of split metacyclic $p$-groups}\label{sec:mainresult}
\noindent 	Here, we present the proof of Theorem \ref{thm:wedderburnsplitmetacyclic}.
	\begin{proof}[Proof of Theorem \ref{thm:wedderburnsplitmetacyclic}.]
		Let $G = \langle a, b\mid a^{p^n} = b^{p^m} = 1, bab^{-1} = a^r \rangle$, where $n\geq 2, m \geq 1$, $(r, p^n) = 1$ and $r$ has multiplicative order $p^s$ modulo $p^n$ such that $1\leq s \leq \min \{n-1, m\}$. Let $\psi \in \Irr(G)$. Suppose $\rho$ is an irreducible $\mathbb{Q}$-representation of $G$ which affords the character $\Omega(\psi)$. Let $A_\mathbb{Q}(\psi)$ be the simple component of the Wedderburn decomposition of $\mathbb{Q}G$ corresponding to $\rho$, which is isomorphic to $M_q(D)$ for some $q\in \mathbb{N}$ and a division ring $D$. Since $m_\mathbb{Q}(\psi) = 1$ (see \cite[Corollary 10.14]{I}), $[D: Z(D)] = m_\mathbb{Q}(\psi)^2$ and $Z(D) = \mathbb{Q}(\psi)$ (see \cite[Theorem 3]{IR}), we have $D = Z(D) = \mathbb{Q}(\psi)$. Now consider $\rho =  \bigoplus_{i=1}^l{\rho_i} $, where $l=[\mathbb{Q}(\psi): \mathbb{Q}]$ and $\rho_i$ is a complex irreducible representation of $G$ affording $\psi^{\sigma_i}$ for some $\sigma_i \in \gal(\mathbb{Q}(\psi)/\mathbb{Q})$. Since $m_\mathbb{Q}(\psi) = 1$, we observe that $q = \psi(1)$ (see \cite[Theorem 2.4]{Y1}).\\
		Let $\psi \in \lin(G)$ and suppose $\rho$ is the irreducible $\mathbb{Q}$-representation of $G$ affording the character $\Omega(\psi)$. Let $\bar{\psi} \in \Irr(G/G')$ such that $\bar{\psi}(gG') = \psi(g)$. Hence, $A_\mathbb{Q}(\psi) \cong \mathbb{Q}(\bar{\psi})$. Moreover, $|\lin(G)|=|\Irr(G/G')|$ and $G/G' = \langle a^{p^s}, b \rangle \cong C_{p^{n-s}}\times C_{p^m}$. Therefore, from Remark \ref{remark:linear_metacyclic}, we have the following two cases.\\
		{\bf Case A ($n-s \geq m$).}  In this case,  the simple components of the Wedderburn decomposition of $\mathbb{Q}G$ corresponding to all irreducible $\mathbb{Q}$-representations of $G$ whose kernels contain $G'$ are
		$$\mathbb{Q} \bigoplus_{\lambda=1}^m (p^\lambda+p^{\lambda-1})\mathbb{Q}(\zeta_{p^\lambda}) \bigoplus_{\lambda=m+1}^{n-s}p^m \mathbb{Q}(\zeta_{p^\lambda})$$
		in $\mathbb{Q}G$.\\
		{\bf Case B ($n-s < m$).} In this case,  the simple components of the Wedderburn decomposition of $\mathbb{Q}G$ corresponding to all irreducible $\mathbb{Q}$-representations of $G$ whose kernels contain $G'$ are
		$$\mathbb{Q} \bigoplus_{\lambda=1}^{n-s} (p^\lambda+p^{\lambda-1})\mathbb{Q}(\zeta_{p^\lambda}) \bigoplus_{\lambda=n-s+1}^{m}p^{n-s} \mathbb{Q}(\zeta_{p^\lambda})$$
		in $\mathbb{Q}G$.
		
		Next, suppose $\psi \in \nl(G)$ and $\rho$ is the irreducible $\mathbb{Q}$-representation of $G$ affording the character $\Omega(\psi)$. Here,  $A_\mathbb{Q}(\psi) \cong M_q(D)$ and $D = \mathbb{Q}(\psi)$. Since $\cd(G) = \{p^t : 0 \leq t \leq s\}$,  $q = \psi(1) = p^t$ for some $1 \leq t \leq s$. We proceed the rest of the proof of Theorem \ref{thm:wedderburnsplitmetacyclic} in the following cases.
		\begin{enumerate}
			\item {\bf Case ($n-s \geq m$).} In this case $n-s > m-t$ for $1 \leq t \leq s$ and hence, $\mathbb{Q}(\psi) = \mathbb{Q}(\zeta_{p^{n-s}})$ (see \eqref{eq:chatractervalue}). Moreover, from Remark \ref{remark:galoisclasses_splitmeta}, for a fix $t$ ($1 \leq t \leq s$), there are $p^{m-t}$ many distinct irreducible $\mathbb{Q}$-representations of $G$ which afford the character $\Omega(\psi)$ such that $\psi \in \Irr(G)$ and $\psi(1)=p^t$. Therefore, for $n-s \geq m$, the simple components of the Wedderburn decomposition of $\mathbb{Q}G$ corresponding to all irreducible $\mathbb{Q}$-representations of $G$ whose kernels do not contain $G'$ are
			$$\bigoplus_{t=1}^s p^{m-t}M_{p^t}(\mathbb{Q}(\zeta_{p^{n-s}}))$$
			in $\mathbb{Q}G$. Hence, along with the above expression and Case A, completes the proof of Theorem \ref{thm:wedderburnsplitmetacyclic}(1).
			\item {\bf Case ($n-s < m$).} Suppose $m = (n-s)+k$. Then we have following two sub-cases.
			\begin{enumerate}
				\item {\bf Sub-case ($k \leq s$).} In this sub-case, $n-s \geq m-t$ for $k \leq t \leq s$ and hence, $\mathbb{Q}(\psi) = \mathbb{Q}(\zeta_{p^{n-s}})$ (see \eqref{eq:chatractervalue}) Thus, from Remark \ref{remark:galoisclasses_splitmeta}, for a fix $t$ (where $t \in \{k, k+1, \dots, s\} $), there are $p^{m-t}$ many distinct irreducible $\mathbb{Q}$-representations of $G$ which afford the character $\Omega(\psi)$ such that $\psi \in \Irr(G)$ and $\psi(1)=p^t$. Therefore, the simple components of the Wedderburn decomposition of $\mathbb{Q}G$ corresponding to all irreducible $\mathbb{Q}$-representations of $G$  which afford the character $\Omega(\psi)$ such that $\psi \in \Irr(G)$ and $\psi(1)=p^t$, where $k \leq t \leq s$ are
				$$\bigoplus_{t=k}^s p^{m-t}M_{p^t}(\mathbb{Q}(\zeta_{p^{n-s}}))$$
				in $\mathbb{Q}G$.\\ Further, in this sub-case, $n-s < m-t$ for $1 \leq t \leq k-1$. Hence from Remark \ref{remark:galoisclasses_splitmeta}, for a fix $t$ (where $t \in \{1, 2, \dots, k-1\} $), there are $\sum_{\lambda = 0}^{n-s}\phi(p^\lambda) = p^{n-s}$ many distinct irreducible $\mathbb{Q}$-representations of $G$ which afford the character $\Omega(\psi)$ such that $\psi \in \Irr(G)$, $\psi(1)=p^t$ and $\mathbb{Q}(\psi)= \mathbb{Q}(\zeta_{p^{n-s}})$, and there are $\phi(p^{n-s})$ many distinct irreducible $\mathbb{Q}$-representations of $G$ which afford the character $\Omega(\psi)$ such that $\psi \in \Irr(G)$, $\psi(1)=p^t$, $\mathbb{Q}(\psi)= \mathbb{Q}(\zeta_{p^{\lambda}})$ with $n-s < \lambda \leq m-t$. Therefore, the simple components of the Wedderburn decomposition of $\mathbb{Q}G$ corresponding to all irreducible $\mathbb{Q}$-representations of $G$  which afford the character $\Omega(\psi)$ such that $\psi \in \Irr(G)$ and $\psi(1)=p^t$, where $1 \leq t \leq k-1$ are
				$$\bigoplus_{t=1}^{k-1} p^{n-s}M_{p^t}(\mathbb{Q}(\zeta_{p^{n-s}})) \bigoplus_{t=1}^{k-1}\bigoplus_{\lambda=n-s+1}^{m-t} (p^{n-s}-p^{n-s-1})M_{p^t}(\mathbb{Q}(\zeta_{p^{\lambda}}))$$
				in $\mathbb{Q}G$. Hence, with the inclusion of the above expressions mentioned in this sub-case and Case B, completes the proof of Theorem \ref{thm:wedderburnsplitmetacyclic}$(2)(a)$.

				\item {\bf Sub-case ($k > s$).} In this sub-case, $n-s < m-t$ for $1 \leq t \leq s$. Hence from Remark \ref{remark:galoisclasses_splitmeta}, for a fix $t$, there are $\sum_{\lambda = 0}^{n-s}\phi(p^\lambda) = p^{n-s}$ many distinct irreducible $\mathbb{Q}$-representations of $G$ which afford the character $\Omega(\psi)$ such that $\psi \in \Irr(G)$, $\psi(1)=p^t$ and $\mathbb{Q}(\psi)= \mathbb{Q}(\zeta_{p^{n-s}})$, and there are $\phi(p^{n-s})$ many distinct irreducible $\mathbb{Q}$-representations of $G$ which afford the character $\Omega(\psi)$ such that $\psi \in \Irr(G)$, $\psi(1)=p^t$, $\mathbb{Q}(\psi)= \mathbb{Q}(\zeta_{p^{\lambda}})$ with $n-s < \lambda \leq m-t$. Therefore, the simple components of the Wedderburn decomposition of $\mathbb{Q}G$ corresponding to all irreducible $\mathbb{Q}$-representations of $G$  whose kernels do not contain $G'$ are
				$$\bigoplus_{t=1}^{s} p^{n-s}M_{p^t}(\mathbb{Q}(\zeta_{p^{n-s}})) \bigoplus_{t=1}^{s}\bigoplus_{\lambda=n-s+1}^{m-t} (p^{n-s}-p^{n-s-1})M_{p^t}(\mathbb{Q}(\zeta_{p^{\lambda}}))$$
				in $\mathbb{Q}G$. 
				Hence, with the inclusion of the above expression mentioned in this sub-case and Case B, completes the proof of Theorem \ref{thm:wedderburnsplitmetacyclic}$(2)(b)$.

			\end{enumerate}
		\end{enumerate}
		This completes the proof of Theorem \ref{thm:wedderburnsplitmetacyclic}.
	\end{proof}
	
\subsection{Examples}\label{exam:splitmetacyclic}	Here, we provide an illustration to find out the Wedderburn decomposition of split metacyclic $p$-groups by using Theorem \ref{thm:wedderburnsplitmetacyclic}.

		\begin{enumerate}
			\item Consider $G_1 = \langle a, b\mid a^{3^4} = b^{3^2} = 1, bab^{-1} = a^{10} \rangle$. Here, $n=4$, $m=2$ and $s=2$ as $r = 10 = 1+3^{4-2}$. Hence, from  Theorem \ref{thm:wedderburnsplitmetacyclic}(1), 
			$$\mathbb{Q}G_1 \cong \mathbb{Q} \bigoplus 4\mathbb{Q}(\zeta_3) \bigoplus 12\mathbb{Q}(\zeta_{9}) \bigoplus 3M_3(\mathbb{Q}(\zeta_{9})) \bigoplus M_9(\mathbb{Q}(\zeta_{9})).$$
			$G_1$ is SmallGroup$(729, 60)$ in GAP library of groups.
			\item Consider $G_2 = \langle a, b\mid a^{3^3} = b^{3^3} = 1, bab^{-1} = a^{4} \rangle$. Here, $n=3$, $m=3$ and $s=2$ as $r = 4 = 1+3^{3-2}$. Hence, from Theorem \ref{thm:wedderburnsplitmetacyclic}(2)(a), 
			$$\mathbb{Q}G_2 \cong \mathbb{Q} \bigoplus 4\mathbb{Q}(\zeta_3) \bigoplus 3\mathbb{Q}(\zeta_{9}) \bigoplus 3\mathbb{Q}(\zeta_{27}) \bigoplus 3M_3(\mathbb{Q}(\zeta_{3})) \bigoplus 2M_3(\mathbb{Q}(\zeta_{9})) \bigoplus 3M_9(\mathbb{Q}(\zeta_{3})).$$
			$G_2$ is SmallGroup$(729, 22)$ in GAP library of groups.
			\item Consider $G_3 = \langle a, b\mid a^{3^2} = b^{3^3} = 1, bab^{-1} = a^{4} \rangle$. Here, $n=2$, $m=3$ and $s=1$ as $r = 4 = 1+3^{2-1}$.Hence, from  Theorem \ref{thm:wedderburnsplitmetacyclic}(2)(b), 
			$$\mathbb{Q}G_3 \cong \mathbb{Q} \bigoplus 4\mathbb{Q}(\zeta_3) \bigoplus 3\mathbb{Q}(\zeta_{9}) \bigoplus 3\mathbb{Q}(\zeta_{27}) \bigoplus 3M_3(\mathbb{Q}(\zeta_{3})) \bigoplus 2M_3(\mathbb{Q}(\zeta_{9})).$$
			$G_3$ is SmallGroup$(243, 21)$ in GAP library of groups.
		\end{enumerate}

	\section{Acknowledgements}
	Ram Karan acknowledges University Grants Commission, Government of India.

\end{document}